\documentclass[11pt]{amsart}
\usepackage{amssymb,amsthm,amsmath}
\usepackage[bookmarks,colorlinks,breaklinks]{hyperref}
\usepackage{color}
\newtheorem{thm}{Theorem}

\newtheorem{prop}[thm]{Proposition}
\newtheorem{lem}[thm]{Lemma}

\newtheorem{question}[thm]{Question}

\theoremstyle{definition}

\numberwithin{thm}{section}
\numberwithin{equation}{thm}

\newcommand{\A}{\ensuremath{{\mathbb{A}}}}
\newcommand{\bP}{\ensuremath{{\mathbb{P}}}}

\renewcommand{\P}{\ensuremath{{\mathbb{P}}}}
\newcommand{\Q}{\ensuremath{{\mathbb{Q}}}}
\newcommand{\bQ}{\ensuremath{{\mathbb{Q}}}}
\newcommand{\Qbar}{\ensuremath{\overline {\mathbb{Q}}}}

\newcommand{\F}{\ensuremath{{\mathbb{F}}}}
\newcommand{\G}{\ensuremath{{\mathbb{G}}}}
\newcommand{\N}{\ensuremath{{\mathbb{N}}}}

\newcommand{\cL}{\mathcal L}

\newcommand{\cO}{\mathcal O}

\newcommand{\lra}{\longrightarrow}
\newcommand{\trdeg}{{\rm trdeg}}
\newcommand{\Kbar}{\ensuremath {\overline K}}
\newcommand{\kbar}{\ensuremath {\overline k}}
\newcommand{\Lbar}{\ensuremath {\overline L}}
\newcommand{\ba}{{\boldsymbol \alpha}}
\newcommand{\bF}{{\boldsymbol f}}

\DeclareMathOperator{\Gal}{Gal}

\DeclareMathOperator{\Aut}{Aut}

\begin{document}

\title{A question for iterated Galois groups in arithmetic dynamics}


\author[A. Bridy]{Andrew Bridy}
\address{Andrew Bridy\\Departments of Political Science and Computer Science\\ Yale University\\
New Haven, CT 06511 \\ USA}
\email{andrew.bridy@yale.edu}

\author[J. R. Doyle]{John R. Doyle}
\address{John R. Doyle \\ Department of Mathematics and Statistics \\ Louisiana Tech University \\ Ruston, LA 71272 \\ USA}
\email{jdoyle@latech.edu}

\author[D. Ghioca]{Dragos Ghioca}
\address{Dragos Ghioca \\ Department of Mathematics \\ University of British Columbia \\ Vancouver, BC V6T 1Z2 \\ Canada}
\email{dghioca@math.ubc.ca}

\author[L.-C. Hsia]{Liang-Chung Hsia}
\address{
Liang-Chung Hsia\\
Department of Mathematics\\
National Taiwan Normal University\\
Taipei, Taiwan, ROC
}
\email{hsia@math.ntnu.edu.tw}

\author[T. J. Tucker]{Thomas J. Tucker}
\address{Thomas J. Tucker\\Department of Mathematics\\ University of Rochester\\
Rochester, NY, 14620, USA}
\email{thomas.tucker@rochester.edu}



\date{}

\dedicatory{}

\begin{abstract}
We formulate a general question regarding the size of the iterated Galois groups associated to an algebraic dynamical system and then we discuss some special cases of our question.
\end{abstract}

\maketitle


\section{Introduction and Statement of Results}\label{intro}


\subsection{A general question}

For any self-map $\Phi$ on a variety $X$ and for any integer $n\ge 0$, we let $\Phi^n$ be the $n$-th iterate of $\Phi$ (where $\Phi^0$ is the identity map, by definition). For a point $x\in X$, we denote by $\cO_\Phi(x)$ the orbit of $x$ under $\Phi$, i.e., the set of all $\Phi^n(x)$ for $n\ge 0$. We say that $x$ is preperiodic if its orbit $\cO_\Phi(x)$ is finite; furthermore, if $\Phi^n(x)=x$ for some positive integer $n$, then we say that $x$ is periodic.

For a projective variety $X$ endowed with an endomorphism $\Phi$, we say that $\Phi$ is polarizable if there exists an ample line bundle $\cL$ on $X$ such that $\Phi^*\cL$ is linearly equivalent to $\cL^{\otimes d}$ for some integer $d>1$. In particular, polarizable endomorphisms are dominant. 

\begin{question}
\label{conj:general}
Let $K$ be a number field or a function field of finite transcendence degree  over a field of characteristic $0$, let $X$ be a smooth, projective variety defined over $K$ and let $\Phi:X\lra X$ be a polarizable endomorphism. 

For each positive integer $n$, we have that $\Phi^n$ induces an inclusion of  the
function field $K(X)$ into itself; we let $G_n(\Phi, X)$ be the Galois group of
the Galois closure of $K(X)$ over itself with respect to this inclusion. We
let $G_\infty:=G_\infty(\Phi,X)$ be the inverse limit of these groups $G_n(\Phi,X)$.

For each point $x \in X(K)$, we let $G_n(\Phi,x)$ be the Galois group of
$K(\Phi^{-n}(x))$ over $K(x)$.  We let $G_{\infty}(x):=G_\infty(\Phi,x)$ be the inverse limit of the
groups $G_n(\Phi,x)$. We have that there is a natural embedding of $G_\infty(x)$ inside $G_\infty$.

Then is it true that at least one of the following statements must hold?
\begin{itemize}
\item[(A)] The index $\left[G_\infty(\Phi,X): G_\infty(\Phi,x)\right]$ is finite. 
\item[(B)] The point $x$ lies in the orbit of a
point in the ramification locus of $\Phi$. 
\item[(C)] The point $x$ lies on a proper subvariety $Y\subset X$ that is
  invariant under a non-identity self-map $\Psi: X \lra X$ with the
  property that $\Phi^n \circ \Psi =\Psi\circ\Phi^n$ for some some
  positive integer $n$. 
\end{itemize}
\end{question}

Jones \cite[Conjecture 3.11]{RafeArborealSurvey} proposes a similar conjecture for
quadratic rational functions $\Phi:\bP^1\lra\bP^1$ over number fields,
but our Question~\ref{conj:general} has not previously been posed
for arbitrary function fields $K$ (of characteristic $0$) or for
self-maps of higher dimensional varieties.

We note first that one needs to exclude the case of finite fields in Question~\ref{conj:general} since generally, the Galois group $G_\infty=G_\infty(\Phi,X)$ is not abelian (even in the case $X=\bP^1$ and $\Phi$ is a rational function), while $G_\infty(x)$ would have to be abelian if $X$ were defined over a finite field $\F_q$ because the Galois group of any  extension of finite fields is abelian. One could ask the same question from Question~\ref{conj:general} under the assumption that $K$ is a finitely generated infinite field (even when its characteristic is positive); however, there are potential complications when the map $\Phi$ is not separable. On the other hand, if $\Phi$ were separable, it is conceivable that one might expect the same conclusion as in Question~\ref{conj:general}.

Secondly, we note that one requires in Question~\ref{conj:general} the assumption that a generic point $x\in X$ has the same (finite) number of preimages under $\Phi$; hence, we could weaken the assumption on the endomorphism $\Phi$ from being polarizable to having the above property and still ask whether the same conclusion as in Question~\ref{conj:general} holds.

Now, in Question~\ref{conj:general}, one can see that if either conclusion~(B)~or~(C) holds, then this is likely to prevent the index $[G_\infty:G_{\infty}(x)]$ from being finite; this is similar to what happens even in the case of rational functions, i.e., $X=\bP^1$ (see \cite[Proposition~3.2]{Quad1}).

In this paper we discuss some special cases of Question~\ref{conj:general} which generally fall outside the scope of the previous studies of the arboreal Galois representation associated to a dynamical system. In particular, we treat the case when $K$ is a function field of transcendence degree greater than one (and $\Phi:\bP^1\lra\bP^1$ is a polynomial mapping), and also we discuss several cases when $X$ is a higher dimensional variety.


\subsection{The case of polynomials defined over function fields of higher transcendence degree}

Here we explain our results towards Question~\ref{conj:general} when $X=\bP^1$ and $\Phi$ is a polynomial mapping (which we denote by $f$), while $K$ is the function field of an arbitrary smooth projective variety defined over $\Qbar$. We start by explaining in more detail the groups appearing in Question~\ref{conj:general} for polynomial mappings $f$.

So, let $K$ be any field, let $f\in K[x]$ with $d=\deg f\geq 2$ and let
$\beta\in \P^1(\Kbar)$.  For $n\in\N$, let
$K_n(f,\beta)=K(f^{-n}(\beta))$ be the field obtained by adjoining the
$n$th preimages of $\beta$ under $f$ to $K(\beta)$. (We declare that
$K(\infty)=K$.)  Set
$K_\infty(f,\beta)=\bigcup_{n=1}^\infty K_n(f,\beta)$.  For
$n\in\N\cup\{\infty\}$, define
$G_n(f,\beta)=\Gal(K_n(f,\beta)/K(\beta))$.  In most of the paper, we
will write $G_n(\beta)$ and $K_n(\beta)$, suppressing the dependence
on $f$ if there is no ambiguity.

The group $G_\infty(\beta)$ embeds into $\Aut(T^d_\infty)$, the
automorphism group of an infinite $d$-ary rooted tree
$T^d_\infty$. Recently there has been much work on the problem of
determining when the index $[\Aut(T^d_\infty):G_\infty(\beta)]$ is
finite (see \cite{BFHJY, BIJJ, BGJT, BT2, BostonJonesArboreal, BostonJonesImage, JKMT, Juul, JonesComp, JonesManes, OdoniIterates, OdoniWreathProducts, PinkQuadratic, PinkQuadraticInfiniteOrbits, RafeArborealSurvey, Quad1}). By work of Odoni~\cite{OdoniIterates}, 
one expects that a generically chosen rational function has a surjective arboreal representation, 
i.e., that $[\Aut(T^d_\infty):G_\infty(\beta)]=1$.

In \cite{Quad1}, the authors studied the family of polynomials $f(x)=x^d+c\in K[x]$, where $K$ is the function field of a smooth, irreducible, projective curve defined over $\Qbar$; note that up to a change of variables, the above polynomials $f(x)$ represent  
all polynomials with precisely one (finite) critical point. Since the field $K$ contains a primitive $d$-th root of unity, 
then it is easy to show that for $f$ in this family, $G_\infty(\beta)$ sits inside $[C_d]^\infty$, 
the infinite iterated wreath product of the cyclic group $C_d$ (with $d$ elements); actually, with the notation as in Question~\ref{conj:general}, we have that $G_\infty=[C_d]^\infty$.  
As $\Aut(T^d_n)\cong [S_d]^n$, this means that if $d\geq 3$, then $[\Aut(T^d_\infty):[C_d]^\infty]=\infty$. Thus
it is impossible for $G_\infty(\beta)$ to have finite index within this family (except when $d=2$). 
However, this simply means that, given the constraint on the size of $G_\infty(\beta)$, 
we should ask a different finite index question. So, the correct problem to study is when $G_\infty(\beta)$ 
has finite index in $[C_d]^\infty=G_\infty$, exactly as predicted by Question~\ref{conj:general}. In our current paper, we extend the results of \cite{Quad1} to the case $K$ is an arbitrary function field defined over $\Qbar$.

So, let $K$ be the function field of a smooth, projective, irreducible variety $V$ over $\Qbar$. We say that $f\in K[x]$ is isotrivial if $f$ is defined over $\Qbar$ 
up to a change of variables, that is, if $\varphi^{-1}\circ f\circ \varphi\in \Qbar[x]$ for 
some $\varphi\in \Kbar[x]$ of degree $1$. In the special case of a unicritical polynomial $f(x)=x^d+c\in K[x]$, we have that $f$ is isotrivial if and only if $c\in\Qbar$. We say that $\beta$ is postcritical for $f$ if $f^n(\alpha)=\beta$ for some $n\ge 1$ and some critical point $\alpha$ of $f$.

The following result is an extension of \cite[Theorem~1.1]{Quad1} to function fields of arbitrary finite transcendence degree and it represents a special case of our Question~\ref{conj:general}.

\begin{thm}\label{p-theorem}
Let $K$ be the function field of a smooth, projective variety defined over $\Qbar$.  
  Let $q = p^{r}$ ($r\ge 1$) be a power of the prime number $p$, let $c\in K\setminus \Qbar$, let
  $f(x) = x^q + c \in K[x]$ and let $\beta\in K$. Then the 
  following are equivalent:
  \begin{enumerate}
  \item The point $\beta$ is neither periodic nor postcritical for $f$.
  \item The group $G_\infty(\beta)$ has finite index in $G_\infty$.
  \end{enumerate}
\end{thm}



It is fairly easy to see that the conditions 
on $\beta$ in Theorem~\ref{p-theorem} are necessary (see \cite[Proposition~3.2]{Quad1}); so, the entire difficulty lies in  showing that these conditions 
are sufficient. Also, we note that the isotrivial case (i.e., $c\in\Qbar$) follows verbatim using the proof from \cite[Section~10]{Quad1}.


As proven in \cite{Quad1}, one of the key steps in the proof of Theorem~\ref{p-theorem} 
is an eventual stability result. As is usual in arithmetic dynamics, we say that the pair $(f,\beta)$ is eventually stable over the field $K$ if 
the number of irreducible $K$-factors of $f^n(x)-\beta$ is uniformly bounded for all $n$. The following result extends \cite[Theorem~1.3]{Quad1} to function fields of arbitrary transcendence degree.

\begin{thm}\label{stable-theorem}
Let $K$ be the function field of a smooth, projective variety defined over $\Qbar$.  Let $q = p^{r}$  ($r\ge 1$) be a power of the prime number $p$.  Let
  $f \in K[x]$ be a polynomial of the form $x^q + c$ where $c\notin \Qbar$. 
  Then for any non-periodic $\beta \in K$, the
  pair $(f,\beta)$ is eventually stable over $K$.  
\end{thm}

We also prove the following disjointness theorem for fields generated
by inverse images of different points under different maps; our result is a generalization of \cite[Theorem~1.4]{Quad1}. 

\begin{thm}\label{disjoint-theorem}
  Let $K$ be the function field of a smooth, projective variety defined over $\Qbar$.    
  For  $i =1, \dots, n$ let $f_i(x) = x^q + c_i \in K[x]$, where
  $c_i\notin \Qbar$, and let $\alpha_i\in K$.  Suppose that there are no
  distinct $i, j$ with the property that $(\alpha_i, \alpha_j)$ lies on a curve in $\A^2$
  that is periodic under the action of $(x,y) \mapsto (f_i(x),
  f_j(y))$.  For each $i$, let $M_i$ denote 
  $K_\infty(f_i,\alpha_i)$.  Then for each $i=1,\dots, n$, we have that 
\[ \left[M_i \cap \left(\prod_{j \ne i} M_j\right): K\right] < \infty .\]
\end{thm}

Theorem~\ref{disjoint-theorem} coupled with Theorem~\ref{p-theorem} yields the following special case of Question~\ref{conj:general}.

\begin{thm}
\label{thm:P1n}
Let $K$ be the function field of a smooth, projective variety defined over $\Qbar$, let $q$ be a power of the prime number $p$, and let $m$ be a positive integer.    
  For  $i =1, \dots, m$ let $f_i(x) = x^q + c_i \in K[x]$, where
  $c_i\notin \Qbar$, and let $\alpha_i\in K$. We let $\underline{\alpha}:=(\alpha_1,\dots, \alpha_m)$ and let $\Phi:=(f_1,\dots, f_m)$ acting on $X:=(\bP^1)^m$. Then let   $G_n(\Phi,\underline{\alpha})$ be the Galois group of
$K(\Phi^{-n}(\underline{\alpha}))$ over $K$.  We let $G_{\infty}(\underline{\alpha}):=G_\infty(\Phi,\underline{\alpha})$ be the inverse limit of the
groups $G_n(\Phi,\underline{\alpha})$. 

Then at least one of the following must hold:
\begin{itemize}
\item[(A)] $\left[G_\infty(\Phi,X): G_\infty(\Phi,\underline{\alpha})\right]$ is finite;
\item[(B)] $\underline{\alpha}$ lies in the orbit of a
point in the ramification locus of $\Phi$; or 
\item[(C)] $\underline{\alpha}$ lies on a proper subvariety $Y\subset X$ that is invariant under a non-identity self-map $\Psi: X \lra X$ with the property that $\Phi\circ \Psi =\Psi\circ\Phi$.
\end{itemize}
\end{thm}

Finally, using a similar strategy as employed in \cite{BT2}, we obtain the following result regarding the arboreal Galois representation associated to cubic polynomials. Once again, we use the notation from Question~\ref{conj:general} for $G_\infty$ and $G_\infty(x)$ and since both groups lie naturally inside $\Aut(T^3_\infty)$, then in order to prove the finiteness of the index of $G_\infty(x)$ inside $G_\infty$, it suffices to prove $[\Aut(T^3_\infty):G_\infty(x)]<\infty$.

\begin{thm}\label{cubic}
  Let $K$ be the function field of a smooth, projective variety
  defined over $\Qbar$.  Let $f \in K[x]$ be a cubic polynomial.
  Assume that $f$ is not isotrivial over
  $\Qbar$, that $\beta$ is not periodic or postcritical for $f$, that
  the pair $(f,\beta)$ is eventually stable, and that $f$ has distinct
  finite critical points $\gamma_1$, and $\gamma_2$, and
  $f^n(\gamma_1)\neq f^n(\gamma_2)$ for all $n\geq 1$.  Then
$$[\Aut(T^3_\infty):G_\infty(\beta)]<\infty.$$
\end{thm}




Our proofs rely mainly on specialization techniques in order to extend the results of \cite{Quad1} and \cite{BT2} to the generality from the current article. Actually, considering the extension of our results from \cite{Quad1} to arbitrary function fields led us to formulating the general Question~\ref{conj:general}.


\subsection{The case of higher dimensional varieties}

Question~\ref{conj:general} was also motivated by the case of the multiplication-by-$m$ maps $\Phi$ (for some integer $m>1$) on abelian varieties $X$. In that case, the conclusion of Question~\ref{conj:general} is known due to Ribet's work \cite{Ribet} (who generalizes results of Bachmakov \cite{Bachmakov}); as long as the point $x\in X$ is not torsion, we know that the index $[G_\infty:G_\infty(x)]$ is finite. Actually, the first result in this direction was the case of the monomial map $\Phi(z):=z^m$ (for integer $m>1$) acting on $\bP^1$, in which case, the conclusion in Question~\ref{conj:general} reduces to the classical theory of Kummerian extensions. 

Similarly, due to work of Pink \cite{Pink-Kummer}, one establishes also the conclusion of Question~\ref{conj:general} in the special case of Drinfeld modules, i.e., if $\Phi$ is a separable additive polynomial (which is, therefore, an endomorphism of $\G_a$ defined over a field of characteristic $p$) of degree larger than $1$ and whose derivative is a transcendental element over $\F_p$. This justifies our belief that as long as $\Phi$ is separable, then Question~\ref{conj:general} should hold as well for finitely generated,  infinite fields of positive characteristic.

In Section~\ref{sec:higher} we present additional evidence supporting our Question~\ref{conj:general}. 


\medskip

{\bf Acknowledgments.} D.G. was partially supported by an NSERC Discovery grant. L.-C. H. was partially supported by 
MOST Grant 108-2115-M-003-005-MY2. 


\section{Wreath products}\label{wreath}


In this section (which overlaps with \cite[Section~2]{Quad1}), we give a brief introduction to wreath products, which arise naturally  
from the Galois theory of the preimage fields $K_n(\beta)=K(f^{-n}(\beta))$.

Let $G$ be a permutation group acting on a set $X$, 
and let $H$ be any group. Let $H^X$ be the group of functions from $X$ to $H$ 
with multiplication defined pointwise, or equivalently the direct product of $|X|$ copies of $H$. 
The wreath product of $G$ by $H$ is the semidirect 
product $H^X\rtimes G$, where $G$ acts on $H^X$ 
by permuting coordinates: for $f\in H^X$ and $g\in G$ we have \[f^g(x)=f(g^{-1}x)\]
for each $x\in X$. We will use the notation $G[H]$ for the wreath product, suppressing the set $X$ 
in the notation. (Another common convention is $H\wr G$ or $H\wr_X G$ if we wish to call attention to $X$.)

Fix an integer $d\geq 2$. For $n\geq 1$, let $T^d_n$ be the complete rooted $d$-ary tree of level $n$. 
It is easy to see that $\Aut(T^d_1)\cong S_d$, and standard to show that $\Aut(T^d_n)$ 
satisfies the recursive formula
\[\Aut(T^d_n)\cong \Aut(T^d_{n-1})[S_d].\]
Therefore we may think of $\Aut(T^d_n)$ as the ``$n$th iterated wreath product" of $S_d$, which we will 
denote $[S_d]^n$. In general, for $f\in K[x]$ of degree $d$ and $\beta\in K$, the Galois group $G_n(\beta)=\Gal(K_n(\beta)/K)$ 
embeds into $[S_d]^n$ via the faithful action of $G_n(\beta)$ on the $n$th level of the tree of preimages of $\beta$ 
(see for example~\cite{OdoniIterates} or~\cite[Section 2]{BT2}). 

Assume now that $f(x):=x^d+c\in K[x]$, where $K$ is a field of characteristic $0$ that contains the $d$-th roots of unity. 
For $\beta\in K$ such that $\beta-c$ is not a $d$th power in $K$, 
we have $K_1(\beta)=K((\beta-c)^{1/d})$ and $G_1(\beta)\cong C_d$. For any $n\geq 2$, 
the extension $K_{n}(\beta)$ is a Kummer extension attained by adjoining to $K_{n-1}(\beta)$ the 
$d$-th roots of $z-c$ where $z$ ranges over the roots of $f^{n-1}(x)=\beta$. Thus we have
\begin{equation*}
\Gal(K_n(\beta)/K_{n-1}(\beta))\subseteq \prod_{f^{n-1}(z)=\beta} \Gal(K_{n-1}(\beta)((z-c)^{1/d})/K_{n-1}(\beta))\subseteq C_d^{d^{n-1}}.
\end{equation*}
This is clear if $f^{n-1}(x)-\beta$ has distinct roots in $\Kbar$. 
If $f^{n-1}(x)-\beta$ has repeated roots, then $\Gal(K_n(\beta)/K_{n-1}(\beta))$ sits inside a direct product of a 
smaller number of copies of $C_d$, so the stated containments still hold. 

Considering the Galois tower
\[ K_n(\beta) \supseteq K_{n-1}(\beta)\supseteq K\]
we see that 
\[G_n(\beta)\subseteq \Gal(K_n(\beta)/K_{n-1}(\beta)) \rtimes G_{n-1}(\beta) \cong G_{n-1}(\beta)[C_d],\]
where the implied permutation action of $G_{n-1}(\beta)$ is on the set of roots of $f^{n-1}(x)-\beta$. By induction,
$G_n(\beta)$ embeds into $[C_d]^n$, the $n$-th iterated wreath product of $C_d$. 
Observe that $[C_d]^n$ sits as a subgroup of $\Aut(T^d_n)\cong[S_d]^n$ via the obvious action on the tree. 
Taking inverse limits, $G_\infty(\beta)$ embeds into $[C_d]^\infty$, which sits as a subgroup of $\Aut(T_\infty)$. 

We summarize our basic strategy for proving that $G_\infty(\beta)$ has finite or infinite index in $[C_d]^\infty$ 
as Proposition~\ref{indexprop} (whose proof is identical with the one for \cite[Proposition~2.1]{Quad1}).

\begin{prop}\label{indexprop}
Let $f=x^d+c\in K[x]$. Then $[[C_d]^\infty:G_\infty(\beta)]<\infty$ if and only if $\Gal(K_n(\beta)/K_{n-1}(\beta))\cong C_d^{d^{n-1}}$ 
for all sufficiently large $n$.
\end{prop}

\section{Heights}\label{heights}


In this section we setup the notation regarding heights. Since later we will need to consider function fields over arbitrary fields (see the proof of Theorem~\ref{disjoint-theorem}, for example), we will introduce the heights associated to  function fields in a general setting; for more details, see \cite{BG}. First, we recall the notation of $\log^+$: for each real number $z$, we have $\log^+|z|:=\log\max\{1,|z|\}$.

So, $K$ is the function field of a smooth, projective (irreducible) variety $V$ defined over a field $K_0$ (of characteristic $0$). As proven in \cite{BG}, there exists a set $\Omega_{V}$ of places of the function field $K/K_0$  associated to the codimension $1$ irreducible subvarieties of $V$; furthermore, there exist positive integers $n_v$ (for each $v\in\Omega_V$) such that the product formula holds for the nonzero elements $z\in K$:
\begin{equation}
\label{eq:prod formula}
\prod_{v\in\Omega_V} |z|_v^{n_v}=1.
\end{equation} 
Then we have the Weil height associated to the set of places in $\Omega_V$; for simplicity, we omit the variety $V$ from the notation for the Weil height and instead, we simply denote 
$$h_{K/K_0}(z)=\sum_{v\in\Omega_V}n_v\cdot \log^+|z|_v.$$
Naturally, the Weil height extends to all points in $\Kbar$ (see \cite{BG}).

We denote by $v(\cdot)$ the (exponential) valuation associated to each place in $\Omega_V$. 
For $f\in K[x]$ with $\deg f=d\geq 2$, 
let $\widehat{h}_f(z)$ be the Call-Silverman canonical height of $z$ relative to $f$~\cite{CallSilverman}, defined by
\[
\widehat{h}_f(z) = \lim_{n\to\infty}\frac{h_{K/K_0}(f^n(z))}{d^n}.
\]

%
%

Finally, we conclude this Section by observing that \cite[Lemmas~3.1~and~3.2]{Quad1} extend to the more general setting of arbitrary function fields with identical proofs.

\section{Eventual Stability} \label{stable}


In this section, we show that if $K$ is a function field over a
finitely generated field of characteristic $0$ and $f \in K[x]$ is a
non-isotrivial unicritical polynomial of degree equal to a prime power, then $f$ is eventually stable. Our results are an extension of the results obtained in \cite[Section~6]{Quad1}.

\begin{prop}\label{stable1}
Let $q=p^r$ (for $r\ge 1$) be a power of the prime number $p$ and let $K$ be a function field of transcendence degree $1$ over a finitely generated field $k$ of characteristic $0$. Let $f(x) = x^q + c \in K[x]$, where $c \notin
k$.  Then for any $\beta\in K$ that is not periodic under $f$, the pair
$(f, \beta)$ is eventually stable over $K$.
\end{prop}

First, we note that since in our Theorem~\ref{stable-theorem} we work under the assumption that the polynomial $x^q+c$ is not defined over $\Qbar$, i.e., $c\notin\Qbar$, then we can readily choose some finitely generated field $k$ such that $K$ is the function field of a curve $C$ defined over $k$ and $c\notin k$ (which is equivalent with $c\notin\kbar$ since $k$ is algebraically closed in $K$).

\begin{proof}[Proof of Proposition~\ref{stable1}.] 
Since $K/k$ is a function field of transcendence degree $1$ and $c\notin k$, then $K$ must be a finite extension of $k(c)$ (and therefore, also a finite extension of $k(c,\beta)$).  Thus, the pair $(f, \beta)$ is eventually stable over
  $K$ if and only if it is eventually stable over $k(c,\beta)$; hence, from now on, 
  we may assume that $K = k(c,\beta)$.  Furthermore, since for each $n$, we have that $f^{-n}(\beta)$ is contained in some algebraic extension of
  $\bQ(c,\beta)$, then we may (and do) assume that $K$ is a finite 
  extension of $\Q(c, \beta)$; note that here we use in an essential way that $k$ is a finitely generated field and therefore, $K\cap \overline{\Q(c,\beta)}$ must be a finite extension of $\Q(c,\beta)$. 

  Now, if $\beta$ is not algebraic over $\Q(c)$, then
  $f^n(x) - \beta$ is easily seen to be irreducible over $\Q(c,\beta)$  
  for all $n$; for example, this follows immediately by looking at its
  Newton polygon at the place at infinity for the function field $\Q(c, \beta)/\Q(c)$.  This would imply that
  $f^n(x) - \beta$ has finitely many factors over $K$; hence, from now on, we may assume that $\Q(c,\beta)$ is an
  algebraic extension of $\Q(c)$.  But then we are back to the setting of \cite[Proposition~6.3]{Quad1} (i.e., we have a function field of transcendence degree one over a number field) and the result follows immediately. 
This concludes our proof of Proposition~\ref{stable1}. 
\end{proof}

\section{Proof of Main Theorems}\label{main proof}

\begin{proof}[Proof of Theorem \ref{stable-theorem}]
We know that $K$ is the function field of a projective, smooth, irreducible variety defined over $\Qbar$; also, we know that the number $c$ (where $f(x)=x^q+c$) is not contained in $\Qbar$. Thus there exists a finitely generated field $k$ such that
\begin{itemize}
\item[(A)] $\trdeg_kK=1$; and 
\item[(B)] $c\notin \kbar$.
\end{itemize}
Furthermore, at the expense of replacing both $k$ and $K$ by finite extensions, we may assume there exists a smooth, projective curve $C$ such that $c,\beta\in L:=k(C)$ and moreover, $K=\Qbar L$. 

By Proposition~\ref{stable1}, the pair $(f,\beta)$ is eventually
  stable over $L$.   It will suffice now to show that 
  $$\Qbar \cap \left(\bigcup_{n=1}^\infty k(f^{-n}(\beta))\right)\text{ is a finite extension of }\Q.$$ 
Since $(f,\beta)$ is eventually stable over $L$, by Capelli's Lemma there exists an
  $m$ such that $f^n(x) - \alpha_i$ is irreducible over $L(\alpha_i)$
  for all $\alpha_i$ such that $f^m(\alpha_i) = \beta$ and all $n\geq m$ 
  (see~\cite[Prop 4.2]{BT2} for a proof of this fact).  Applying
  \cite[Proposition~7.7]{Quad1} (which extends verbatim to our setting if $s=1$ in \cite[Proposition~7.7]{Quad1}) and also applying \cite[Proposition~8.1]{Quad1} (again in the special case $s=1$, which does not require the results of \cite[Section~5]{Quad1}), we see then that there is an integer $n_1$
  such that for all $n > n_1$, the field $L_n(\beta):=L(f^{-n}(\beta))$ contains no
  nontrivial extensions of $L_{n-1}(\beta)$ that are unramified
  over $L_{n-1}(\beta)$.  Let $\gamma$ be any element of
  $\Qbar \cap L_{\ell}(\beta)$ for some $\ell$.  Let $N$ be minimal among all
  integers such that $\gamma \in L_N(\beta)$.  Since
  $L_{N-1}(\beta)(\gamma)$ is unramified over
  $L_{N-1}(\beta)$, it follows that $N \leq n_1$.  Hence
  $\Qbar \cap \bigcup_{n=1}^\infty L(f^{-n}(\beta))$ is a finite extension of $L$, as
  desired. 
\end{proof}

\begin{proof}[Proof of Theorem \ref{p-theorem}]
As proven in \cite[Proposition~3.2]{Quad1}, we already know that the conditions are necessary. Therefore assume 
 that $\beta$ is not postcritical nor periodic for $f$. 
 By Theorem \ref{stable-theorem}, the pair $(f,\beta)$ is eventually
 stable.  Again using Capelli's Lemma, there is some $m$ such that for all $\alpha\in f^{-m}(\beta)$ and for all $n\geq 1$, 
 $f^n(x)-\alpha$ is irreducible over $K_m(\beta)$. By \cite[Proposition~7.7]{Quad1} and also using \cite[Proposition~8.1]{Quad1} (which are both valid in our setting  in their special case when $s=1$, which does not require the results of \cite[Section~5]{Quad1}),   
 there exists $n_0$ such that for all $n\geq n_0$, we have
 \[\Gal(K_n(\beta)/K_{n-1}(\beta))\cong C_q^{q^{n-1}}.\]
 By Proposition~\ref{indexprop}, we are done.
\end{proof}

 \begin{proof} [Proof of Theorem \ref{disjoint-theorem}]
We note that Theorem~\ref{disjoint-theorem} was proven in \cite[Theorem~1.4]{Quad1} when $K$ is a function field of transcendence degree $1$ over a number field. So, from now on, assume $K$ has transcendence degree  $m>1$ (over a number field). Then (at the expense of replacing $K$ by a finite extension) there exists a finite tower of field extensions $L_0\subset L_1\subset \cdots \subset L_m=K$, where each $L_i/L_{i-1}$ is a function field of (relative) transcendence degree $1$ and furthermore, $L_0$ is a number field. Our strategy is to show that given any function field $K/L$  of transcendence degree $1$ (where $L$ itself is a function field over another field $F$) with the property that the $\alpha_i$'s and the $c_i$'s satisfy the hypotheses of Theorem~\ref{disjoint-theorem}, then there exists a specialization at a place $\gamma$ for the function field $K/L$ with the property that the corresponding $\alpha_i(\gamma)$'s and the $c_i(\gamma)$'s still satisfy the same hypotheses. So, after finitely many suitable specializations, we descend our problem to the case when all the points are contained in a function field of transcendence degree $1$ over a number field and therefore, the desired conclusion is delivered by \cite[Theorem~1.4]{Quad1}.

First of all, we let $h_{K/L}$ be the Weil height for the function field $K=L(C)$ where $C$ is a projective, smooth curve defined over $L$. We assume the $\alpha_i$'s and the $c_i$'s verify the hypotheses of Theorem~\ref{disjoint-theorem}. Also, we let $h_C:C(\Lbar)\lra \mathbb{R}_{\ge 0}$ be the Weil height associated to a given ample divisor on $C$. For each polynomial $f\in K[x]$ of degree larger than one, we have the canonical height $\widehat{h}_f$ associated to the polynomial $f$. Also, since $L$ itself is a function field over some other field $F$, then for each point $z$ of $\Lbar$, we let $h_{L/F}(z)$ be its Weil height computed with respect to the function field $L/F$.

For each $\gamma\in C(\Lbar)$ for which the coefficients of $f(x)$ are well-defined at $\gamma$, we let $f_\gamma$ be the specialization of $f$ at $\gamma$; in our proof, we will work with $f(x):=x^q+c$ for $c\in K$ and thus (viewing $c\in L(C)$), the specialization $f_\gamma(x):=x^q+c(\gamma)$ is constructed for points $\gamma\in C(\Lbar)$ such that $c(\gamma)$ is well-defined; clearly, the $c_i$'s and the $\alpha_i$'s are well-defined at all but finitely many specializations for the function field $K/L$. Also, for all but finitely many specializations, we have that
\begin{itemize}
\item[(i)] if for a pair of distinct indices $i\ne j$, we have that $c_i/c_j$ is not a $(q-1)$-st root of unity, then also $c_i(\gamma)/c_j(\gamma)$ is not a $(q-1)$-st root of unity. 
\end{itemize} 

\begin{lem}
\label{lem:0spec}
There exists $M>0$ such that whenever $h_C(\gamma)>M$, we have that each $c_i(\gamma)\notin\Qbar$.
\end{lem}

\begin{proof}[Proof of Lemma~\ref{lem:0spec}.]
Clearly, it suffices to prove this statement for one $c_i$, which we will denote by $c$. Now, if $c\in \Lbar$, then any specialization works as desired. On the other hand, if $c\notin\Lbar$, then $h_{K/L}(c)>0$ and using \cite{CallSilverman}, we have that $\lim_{h_C(\gamma)\to\infty} h_{L/F}(c(\gamma))/h_C(\gamma)=h_{K/L}(c)>0$, thus proving that for some $M>0$, whenever $h_C(\gamma)>M$ then $h_{L/F}(c(\gamma))>0$, which yields that $c(\gamma)\notin \Qbar$ (since those points would have height equal to $0$ for the function field $L/F$).
\end{proof}

Next we show that for all specializations of sufficiently large height, 
\begin{itemize}
\item[(ii)] there are no distinct indices $i\ne j$ such that $(\alpha_i(\gamma), \alpha_j(\gamma))$ lies on a periodic curve under the action of $(f_{i,\gamma}, f_{j,\gamma})$.
\end{itemize}
We achieve this goal by combining Lemmas~\ref{lem:specialization1}~and~\ref{lem:specialization2}. 
\begin{lem}
\label{lem:specialization1}
Let $K=L(C)$ be the function field of a curve (where $L/F$ is itself a function field), let $c\in K$, let $f(x):=x^q+c$ and let $\beta_1,\beta_2\in K$ such that there is no integer $n\ge 0$ with the property that $f^n(\beta_1)=\beta_2$. If either $c$ or $\beta_1$ is not contained in $\Lbar$, then there exists a real number $M_0>0$  such that for all points $\gamma\in C(\Lbar)$ satisfying $h_C(\gamma)>M_0$, we have that there exists no integer $n\ge 0$ such that $f_\gamma^n(\beta_1(\gamma))=\beta_2(\gamma)$.
\end{lem}

\begin{proof}[Proof of Lemma~\ref{lem:specialization1}.]
First we note that if $\beta_1$ is preperiodic, then its orbit under $f$ is finite and therefore, away from finitely many points of $C$, the specialization of $\beta_2$ will avoid the corresponding specialization of a point in the orbit of $\beta_1$ under $f$. So, from now on, we assume $\beta_1$ is not preperiodic.

Since not both $c$ and $\beta_1$ are contained in $\Lbar$, then we get that $\widehat{h}_f(\beta_1)>0$. Indeed, we know that $\beta_1$ is not preperiodic and therefore, according to \cite{Ben1}, either $\widehat{h}_f(\beta_1)>0$ or the pair $(f,\beta_1)$ is isotrivial for the function field $K/L$, which in our case is equivalent with both $c$ and $\beta_1$ being contained in $\Lbar$ (note that either $f(x)=x^q+c$ is defined over $\Lbar$, or no conjugate of $f$ under a linear transformation would be defined over $\Lbar$). So, our hypothesis for Lemma~\ref{lem:specialization1} yields that $\widehat{h}_f(\beta_1)>0$. Then \cite{CallSilverman} yields that
\begin{equation}
\label{eq:spec Sil}
\lim_{h_C(\gamma)\to\infty} \frac{\widehat{h}_{f_\gamma}(\beta_1(\gamma))}{h_C(\gamma)}=\widehat{h}_f(\beta_1)>0.
\end{equation} 
In particular, there exist positive real numbers $C_0$ and $M_0$ such that that if $h_C(\gamma)>M_0$, then 
\begin{equation}
\label{eq:spec Sil 1.5}
\widehat{h}_{f_\gamma}(\beta_1(\gamma))>C_0\cdot h_C(\gamma). 
\end{equation}
Again using \cite{CallSilverman} (see also \cite{SilvermanADS}), we get that there exist positive real numbers $C_1$ and $C_2$ such that
\begin{equation}
\label{eq:spec Sil 2}
h_{L/F}(\beta_2(\gamma))\le C_1\cdot h_C(\gamma)+C_2.
\end{equation} 
Using \cite{SilvermanADS} (see also \cite[Proposition~2.4]{variation_paper}), we get that there exist positive constants $C_3$ and $C_4$ such that for each $z\in \Lbar$, we have 
\begin{equation}
\label{eq:spec Sil 3}
h_{L/F}(z)>\widehat{h}_{f_\gamma}(z)-C_3\cdot h_C(\gamma)-C_4.
\end{equation}
Now, let $\gamma\in C(\Lbar)$ such that $h_C(\gamma)>M_0$ and assume there exists some nonnegative integer $n$ such that $f_\gamma^n(\beta_1(\gamma))=\beta_2(\gamma)$. Using \eqref{eq:spec Sil 3} and the definition of the canonical height, we get
\begin{equation}
\label{eq:spec Sil 4}
h_{L/F}\left(f_\gamma^n(\beta_1(\gamma))\right)> \widehat{h}_{f_\gamma}\left(f_\gamma^n(\beta_1(\gamma))\right) - C_3h_C(\gamma)-C_4=q^n\widehat{h}_{f_\gamma}(\beta_1(\gamma))-C_3h_C(\gamma)-C_4.
\end{equation}
Using inequality \eqref{eq:spec Sil 1.5} in \eqref{eq:spec Sil 4}, we get
\begin{equation}
\label{eq:spec Sil 5}
h_{L/F}\left(f_\gamma^n(\beta_1(\gamma))\right)>(q^nC_0-C_3)\cdot h_C(\gamma)-C_4. 
\end{equation}
Now, combining inequalities \eqref{eq:spec Sil 2} and \eqref{eq:spec Sil 5} along with the equality $f_\gamma^n(\beta_1(\gamma))=\beta_2(\gamma)$, we obtain  that
\begin{equation}
\label{eq:spec Sil 6}
(q^nC_0-C_3-C_1)\cdot h_C(\gamma)< C_2+C_4,
\end{equation}
which yields that if $h_C(\gamma)>M_0$, then $$n<\log_q\left(\frac{\frac{C_2+C_4}{M_0}+C_1+C_3}{C_0}\right).$$
On the other hand, for each of the finitely many nonnegative integers $n$ satisfying the above inequality, there exist finitely many points $\gamma\in C(\Lbar)$ such that $f_\gamma^n(\beta_1(\gamma))= \beta_2(\gamma)$. So, at the expense of replacing $M_0$ by a larger number,  we obtain the conclusion from Lemma~\ref{lem:specialization1}. 
\end{proof}

Next we note that if we are to assume that $c$, $\beta_1$ and also $\beta_2$ are contained in $\Lbar$, then for \emph{every} specialization, we get that these elements are unchanged through the corresponding specialization and therefore, if originally, there was no integer $n$ such that $f^n(\beta_1)=\beta_2$ (where $f(x):=x^q+c$), then this conclusion remains valid for \emph{each} specialization. So, we are left to analyze the case when $c,\beta_1\in\Lbar$, while $\beta_2\in K\setminus \Lbar$.

\begin{lem}
\label{lem:specialization2}
Let $c,\beta_1\in \Lbar$, let $f(x):=x^q+c$ and let $\beta_2\in K\setminus\Lbar$. If $c\notin\Qbar$, then the following statements hold: 
\begin{itemize}
\item[(a)] If $\beta_1$ is preperiodic under the action of $f$, then for all but finitely many specializations $\gamma$, there is no integer $n$ such that $f^n(\beta_1)=\beta_2(\gamma)$. 
\item[(b)] If $\beta_1$ is not preperiodic under the action of $f$, then there exist positive constants $C_8$ and $C_9$ such that for any nonnegative integer $n$ and for any $\gamma\in C(\Lbar)$, if $f^n(\beta_1)=\beta_2(\gamma)$ then 
$$\left|h_C(\gamma)-C_8\cdot q^n\right|\le C_9.$$
\end{itemize} 
\end{lem}

\begin{proof}[Proof of Lemma~\ref{lem:specialization2}.] 
The proof of part~(a) is identical with the corresponding case ($\beta_1$ is preperiodic under the action of $f$) from Lemma~\ref{lem:specialization1}. So, from now on, we assume $\beta_1$ is not preperiodic under the action of $f$. 

Now, since $c\notin \Qbar$, we have that the polynomial $f(x)=x^q+c$ is not isotrivial for the function field $L/\Q$ and so, \cite{Ben1} yields that $C_5:=\widehat{h}_{f,L/\Q}(\beta_1)>0$ (where $\widehat{h}_{f,L/\Q}$ is the canonical height of the polynomial $f\in \Lbar[x]$ constructed with respect to the height for the function field $L/\Q$, which is a function field of finite transcendence degree, most likely larger than one). As in the proof of Lemma~\ref{lem:specialization1}, we have (see \cite{SilvermanADS}) that 
\begin{equation}
\label{eq:spec Sil 7}
\left|h_{L/\Q}(\beta_2(\gamma)) -C_1\cdot h_C(\gamma)\right| \le C_2
\end{equation} 
for some positive constants $C_1$ and $C_2$; note that $C_1>0$ since $\beta_2\notin\Lbar$. Also, there exists a positive constant $C_7$ (see \cite{CallSilverman}) such that for each $z\in \Lbar$ we have  
\begin{equation}
\label{eq:spec Sil 8}
\left|h_{L/\Q}(z)-\widehat{h}_{f, L/\Q}(z)\right|\le C_7.
\end{equation}
So, if $f^n(\beta_1)=\beta_2(\gamma)$ for some nonnegative integer $n$, then equations \eqref{eq:spec Sil 7} and \eqref{eq:spec Sil 8} (along with the fact that $\widehat{h}_{f,L/\Q}(f^n(\beta_1))=q^n\cdot C_5$) yield that
\begin{equation}
\label{eq:spec Sil 9}
\left|C_1h_C(\gamma)-q^nC_5\right|\le C_2+C_7.
\end{equation}
Then taking $C_8:=\frac{C_5}{C_1}$ and $C_9:=\frac{C_2+C_7}{C_1}$, we see that inequality \eqref{eq:spec Sil 9} provides the desired conclusion from part~(b) of Lemma~\ref{lem:specialization2}. 
\end{proof}

Now we explain how to combine Lemmas~\ref{lem:0spec},~\ref{lem:specialization1}~and~\ref{lem:specialization2} to provide the conclusion~(ii) from above for some suitable specialization at a point $\gamma\in C(\Lbar)$. First, we notice that since no $c_i\in\Qbar$, then (in particular) $c_i\ne 0, -2$ (note that $c_i=0$ yields a monomial function $f_i(x)$, while $c_i=-2$ and $d=2$ yields the second Chebyshev polynomial); hence, in the language from \cite{MedvedevScanlon} (see also \cite{G-Adv, GN-IMRN, G-JEMS}), the polynomials $f_i(x)=x^q+c_i$ are disintegrated, or non-special (i.e., not conjugated to monomials or Chebyshev polynomials). Also, the hypothesis of Theorem~\ref{disjoint-theorem} yields that no $\alpha_i$ can be periodic for the corresponding polynomial $f_i$. Now, \cite[Proposition~7.7]{G-JEMS} yields that if there exists a plane curve, projecting dominantly onto each coordinate, which is periodic under the action of $(x,y)\mapsto (f_i(x), f_j(y))$, then there must be some $(q-1)$-st root of unity $\zeta$ such that $c_j=\zeta\cdot c_i$. Furthermore, as proven in \cite{MedvedevScanlon} as a result of a deep analysis of polynomial decompositions along with a powerful study of the model theory of algebraically closed field with a distinguished automorphism \emph{ACFA} (see also \cite[Proposition~2.5]{G-Adv} and \cite[Proposition~5.5]{GN-IMRN}), assuming $c_j=\zeta c_i$, we have that each  plane curve (projecting dominantly onto each coordinate) which is periodic under the action of $(x,y)\mapsto (f_i(x), f_j(y))$ must be of the form
\begin{equation}
\label{eq:spec 11}
y=\zeta\cdot f_i^n(x)\text{ for some }n\ge 0
\end{equation}
or
\begin{equation}
\label{eq:spec 12}
x=\zeta^{-1}\cdot f_i^n(y)\text{ for some }n\ge 0.
\end{equation}
Using the fact that the only periodic curves under the action of $(x,y)\mapsto (f_i(x), f_j(x))$ are the ones described by equations \eqref{eq:spec 11} and \eqref{eq:spec 12} and that can only happen if $c_j=\zeta\cdot c_i$, then Lemmas~\ref{lem:specialization1}~and~\ref{lem:specialization2} yield that any point $\gamma\in C(\Lbar)$ for which $f_C(\gamma)$ is sufficiently large (see Lemma~\ref{lem:specialization1}) and, furthermore, which also satisfies the property 
\begin{equation}
\label{eq:spec 13}
h_C(\gamma)\notin  \left[C_8q^n-C_9, C_8q^n+C_9\right]\text{ for all }n\ge 0
\end{equation}
for some suitable positive constants $C_8$ and $C_9$ (see Lemma~\ref{lem:specialization2}) 
would induce a specialization which satisfies the conditions~(i)-(ii) above. Easily, we see that there exist infinitely many such suitable specializations. 


Now, for such a suitable specialization at a point $\gamma\in C(\Lbar)$, we note that specializing at $\gamma$ a field extension $K_1\subset K_2$ (which are themselves finite field extensions of $K$) yields a finite field extension $L_1\subset L_2$ (which are themselves finite extensions of $L$) and moreover, 
\begin{equation}
\label{eq:spec 20}
[L_2:L_1]\le [K_2:K_1].
\end{equation}
So, following the proof of Theorem~\ref{disjoint-theorem} for the specializations of $f_i$ and of $\alpha_i$ at $\gamma$ yields that there is a (minimal) integer $n_2$ such for all $n > n_2$, we have
\begin{equation}
\label{eq:spec 21}
\Gal(L_{n}(\bF_\gamma,\ba(\gamma))/ L_{n-1}(\bF_\gamma,\ba(\gamma))) \cong C_q^{s q^{n-1}},
\end{equation}
where $L_n$ is the specialization at $\gamma$ of $K_n$ (while $\bF_\gamma$ and $\ba(\gamma)$ are the respective specializations of $\bF$ and $\ba$ at $\gamma$). 
Then combining \eqref{eq:spec 20} with \eqref{eq:spec 21} yields that 
$$\Gal(K_{n}(\bF,\ba)/ K_{n-1}(\bF,\ba)) \cong C_q^{s q^{n-1}},$$
and then the rest of the proof of Theorem~\ref{disjoint-theorem} follows verbatim as in the proof of \cite[Theorem~1.4]{Quad1}.  
\end{proof}

\begin{proof}[Proof of Theorem~\ref{thm:P1n}.]
Assume that condition~(B) does not hold. Then for each $i=1,\dots, m$, we have that $\alpha_i$ is not a postcritical point (i.e., iterate of $0$ under the map $f_i$). Now, assume that also condition~(C) does not hold; in particular, this means that no $\alpha_i$ is a periodic point for the corresponding map $f_i$. Using the fact that $\alpha_i$ is neither periodic nor postcritical, Theorem~\ref{p-theorem} yields that each Galois group $G_\infty(f_i,\alpha_i)$ has finite index inside the corresponding $G_\infty(f_i,\bP^1)$. 

Furthermore, since condition~(C) does not hold, then there are no
  distinct $i, j$ with the property that $(\alpha_i, \alpha_j)$ lies on a plane curve 
  that is periodic under the action of $(x,y) \mapsto (f_i(x),
  f_j(y))$. So, letting $M_i$ be $K_\infty(f_i,\alpha_i)$,  then for each $i=1,\dots, m$, we have that 
\begin{equation}
\label{eq:field_disjoint} 
\left[M_i \cap \left(\prod_{j \ne i} M_j\right): K\right] < \infty ,
\end{equation}
by Theorem~\ref{disjoint-theorem}. Using \eqref{eq:field_disjoint}, we obtain that $G_\infty(\Phi,\underline{\alpha})$ has finite index in $\prod_{i=1}^m G_\infty(f_i,\alpha_i)$ and then combining this information with the fact that $G_\infty(f_i,\alpha_i)$ has finite index in $G_\infty(f_i,\bP^1)$, we obtain the desired conclusion from Theorem~\ref{thm:P1n}.
\end{proof}

\begin{proof}[Proof of Theorem \ref{cubic}.]
At the expense of replacing $f(x)$ with a conjugate of it, we may assume that  $f(x):=x^3 + bx + a$.  Since $f^{-n}(\beta)$ is
  algebraic over $\Q(a,b,\beta)$ for all $n$, we may assume then that
  $K$ is a finite extension of $\Q(a,b,\beta)$.  In the case where $K$
  has transcendence degree 1 over $\Q$, then \cite[Theorem 1.1]{BT2}
  states $[\Aut(T^3_\infty):G_\infty] < \infty$.  If $\beta$ is not
  algebraic over $\Q(a,b)$, then \cite[Proposition 12.1]{BT2} gives the
  even stronger result that $[\Aut(T_\infty):G_\infty] =1 $.

  Thus, we are left with treating the case where $a$ and $b$ are
  algebraically independent and $\beta$ is algebraic over $\Q(a,b)$.
  We treat this case by specializing from $K$ to a finite extension of
  $\Q(b)$, so that we may apply  \cite[Theorem 1.1]{BT2}; in particular, we will work with specializations $t$ such
  that $a_t = G(b)$ where $G$ is a polynomial with positive integer
  coefficients.  Note that when $f_t(x) := x^3 + bx + G(b)$, for $G$ an even 
  polynomial whose nonzero coefficients are positive integers (say, $G(b)=b^4+3b^2+5$), $f_t$ must be
  eventually stable, since $f$ can be further specialized to a
  polynomial of the form $x^3 + 3mx + n$ where $m, n$ are integers,
  and such polynomials are known to be eventually stable over any
  finite extension of $\Q$ by \cite{RafeAlon}.  Furthermore, we then
  have $f_t^i((\gamma_1)_t) \not= f_t^i((\gamma_2)_t)$ for all $i$
  since sending $b$ to $-3e^2$, for $e$
  a positive integer, yields critical points $\pm e$ and
  $f_t^i(-e) > f_t^i(e)$ for all $i$ for a polynomial $G$ as above (whose only nonzero terms have even degree and positive integer coefficients).  Similarly, if $\deg G(b)$ is larger
  than two then neither $(\gamma_1)_t$ nor $(\gamma_2)_t$ can be
  preperiodic, since the heights of the iterates of each must go to
  infinity.  

  Now, by \cite[Theorem 4.1]{CallSilverman}, we have
\begin{equation}\label{hb}
  \lim_{h(t) \to \infty}
  \frac{\widehat{h}_{f_t}(\beta)}{h(t)} = \widehat{h}_f(\beta),
\end{equation}
where $\widehat{h}_f$ is the canonical height for the polynomial $f$ with respect to the heights on the function field $\Q(a,b)/\Q(b)$, while $\widehat{h}_{f_t}$ is the canonical height of the specialization polynomial $f_t$ with respect to the heights for the function field $\Q(b)/\Q$; also, $h(t)$ refers to the Weil height for the curve $\bP^1_{\Q(b)}$ so that the height of the point when we specialize $a_t:=G(b)$ is simply the degree of the polynomial $G$.
  
If $\beta$ is not
  preperiodic, then since $f$ is not isotrivial over $\Q(b)$, we have $\widehat{h}_f(\beta)
  > 0$ by \cite{Ben1}, so $\widehat{h}_{f_t}(\beta_t) > 0$ when $h(t)$ is large.  If $\beta$ is
  preperiodic, then there are at most finitely many specializations
  $t$ such that $\beta_t$ is periodic, since $\beta$ itself is not
  periodic.  Thus, in either event, $\beta_t$ is not periodic for all $t$
  of sufficiently large height.  We also have
  \begin{equation}\label{hg}
  \lim_{h(t) \to \infty}
  \frac{\widehat{h}_{f_t}(\gamma_i)}{h(t)} = \widehat{h}_f((\gamma_i)_t) > 0,
\end{equation}
again by \cite[Theorem 4.1]{CallSilverman} and \cite{Ben1}.  Thus,
choosing a $j$ such that $3^j \widehat{h}_f(\gamma_i) > \widehat{h}_f(\beta)$, for $i=
1, 2$, we see that for all specializations $t$ of sufficiently large height, we have
$\widehat{h}_{f_t}(f_t^N((\gamma_i)_t)) > \widehat{h}_{f_t}(\beta_t)$ for $i=1,2$ for all
$N \geq j$.  Since there are at most finitely many $t$ such that
$f_t^m((\gamma_i)_t) = \beta_t$ for $i=1,2$ and $m < j$, we see again
that for all $t$ of sufficiently large height, we have
$f_t^n((\gamma_i)_t) \not= \beta_t$ for all $n$.

Since we may choose a specialization $t$ such that $a_t = G(b)$, for
$G$ a polynomial of arbitrarily large degree, then there are specializations $t$ with $h(t)$ arbitrarily
large such that $f_t$ has the desired form $x^3 + bx + G(b)$ (where $G$ is a polynomial whose nonzero terms have even degrees and positive integer coefficients).  Hence, there is a
specialization $t$ such that $f_t$ is eventually stable, $\beta_t$ is
not post-critical and not periodic,  the critical points
$(\gamma_1)_t, (\gamma_2)_t$ of $f_t$ are not periodic, and we do not have
$f_t^i((\gamma_1)_t) = f_t^i((\gamma_2)_t)$ for any positive integers
$i$.  Thus, the pair $(f_t, \beta_t)$ satisfies the hypotheses of
\cite[Theorem 1.1]{BT2}.  Now, for all $n$ the degree $[K(f^{-n}(\beta)):K]$ is at
least as large as $[k_t(f_t^{-n}(\beta_t)): k_t]$, where $k_t$ is the
field of definition of the point $t$, so  we must have
$[\Aut(T^3_\infty):G_\infty] < \infty$, as desired.  
\end{proof}


\section{The higher dimensional case for our main question}
\label{sec:higher}

It makes sense to consider the following more general question, which provides a geometric extension to Question~\ref{conj:general}. So, given a polarizable endomorphism $\Phi$ of a smooth, projective variety $X$ defined over a field $K$ of characteristic $0$, we let $G_\infty:=G_\infty(\Phi)$ be defined as in Question~\ref{conj:general} as the inverse limit of the Galois groups for the Galois closures of $K(X)/(\Phi^n)^*K(X)$ (as $n$ goes to infinity). Now, let $Z\subset X$ be a proper (closed, irreducible) subvariety and let $G_n(\Phi,Z)$ be the Galois group for Galois closure of the cover $\Phi^{-n}(Z)\lra Z$; then let $G_\infty(Z):=G_\infty(\Phi,Z)$ be the inverse limit of all these groups $G_n(\Phi,Z)$. Similar to Question~\ref{conj:general}, one expects that at least one of the following three possibilities must hold:
\begin{itemize}
\item[(A')] $[G_\infty:G_\infty(Z)]$ is finite;
\item[(B')] $\Phi^{-n}(Z)$ does not intersect properly with the ramification locus of $\Phi$, for some $n\ge 0$; or 
\item[(C')] there exists a proper subvariety $Y\subset X$ containing $Z$, which is invariant under the action of a non-identity self-map $\Psi:X\lra X$ with the property that $\Psi\circ\Phi^n=\Phi^n\circ \Psi$ for some positive integer $n$.
\end{itemize}

In the case when Question~\ref{conj:general} is known, then also the above question has positive answer. Indeed, if $\dim(X)=1$, then the above question is precisely the problem investigated by Question~\ref{conj:general}. On the other hand, when $X$ is an arbitrary abelian variety endowed with the multiplication-by-$m$ map $\Phi$, condition (B') above is vacuous since there is no ramification for $\Phi$. So, in this case, we have that either the index $[G_\infty:G_\infty(Z)]$ is finite, or $Z$ is contained in a proper algebraic subgroup of $X$, which is essentially the content of (C'). 

Furthermore, for any polarizable dynamical system $(X,\Phi)$, given any subvariety $Z\subset X$, it is sufficient to find one point $x$ inside $Z$ which does not satisfy conditions~(B)-(C) from Question~\ref{conj:general}; then, according to Question~\ref{conj:general}, we have that $[G_\infty:G_\infty(x)]<\infty$. Indeed, for any point $x\in Z$, we have that $G_\infty(x)$ is the decomposition group of $G_\infty(Z)$ corresponding to the prime associated to the closed point $x$ on $Z$ (this fact can be established by analyzing the statement at each level $n$ and then taking inverse limits).

So, as explained in the previous paragraph, as long as the closed subvariety $Z$ contains a point $x$ which verifies the conditions~(B)-(C) from Question~\ref{conj:general}, then we would expect that also $[G_\infty:G_\infty(Z)]<\infty$. On the other hand, one expects to find such a point $x\in Z$ as long as $Z$ does not satisfy conditions~(B')-(C'). For example, condition~(C') yields that $Z$ is not preperiodic under the action of $\Phi$ and therefore, assuming the Dynamical Manin-Mumford Conjecture holds (see \cite{Zhang}), there should be a Zariski dense set of non-preperiodic points inside $Z$. Actually, the stronger assumption from (C') which refers to the action on $Z$ by any self-map $\Psi$ commuting with $\Phi$ should yield - coupled with the Dynamical Manin-Mumford Conjecture - the existence of a point $x\in Z$ satisfying condition~(C), as predicted by \cite{GTZ-IMRN, Yuan-Zhang, G-Compositio}.

\bibliographystyle{amsalpha}
\bibliography{QuadBib}


\end{document}